\documentclass[12pt]{amsart}
\usepackage{amssymb}
\usepackage{graphics}
\usepackage{latexsym}
\usepackage{amsmath}
\usepackage{amssymb,amsthm,amsfonts}
\usepackage{amscd}
\usepackage[arrow, matrix, curve]{xy}
\usepackage{syntonly}
\title[Characteristic $p$ nonabelilan Hodge theory in geometric case]{A note on the characteristic $p$ nonabelian Hodge theory in the geometric case}
\author[Mao Sheng]{Mao Sheng}
\email{msheng@ustc.edu.cn}
\address{School of Mathematical Sciences,
University of Science and Technology of China, Hefei, 230026, China}
\author[He Xin]{He Xin}
\author[Kang Zuo]{Kang Zuo}
\email{xinhe@uni-mainz.de} \email{zuok@uni-mainz.de}
\address{Institut f\"{u}r  Mathematik, Universit\"{a}t
Mainz, Mainz, 55099, Germany}

\begin{document}
\theoremstyle{plain}
\newtheorem{thm}{Theorem}[section]
\newtheorem{theorem}[thm]{Theorem}
\newtheorem{lemma}[thm]{Lemma}
\newtheorem{corollary}[thm]{Corollary}
\newtheorem{proposition}[thm]{Proposition}
\newtheorem{addendum}[thm]{Addendum}
\newtheorem{variant}[thm]{Variant}
\theoremstyle{definition}
\newtheorem{lemma and definition}[thm]{Lemma and Definition}
\newtheorem{construction}[thm]{Construction}
\newtheorem{notations}[thm]{Notations}
\newtheorem{question}[thm]{Question}
\newtheorem{problem}[thm]{Problem}
\newtheorem{remark}[thm]{Remark}
\newtheorem{remarks}[thm]{Remarks}
\newtheorem{definition}[thm]{Definition}
\newtheorem{claim}[thm]{Claim}
\newtheorem{assumption}[thm]{Assumption}
\newtheorem{assumptions}[thm]{Assumptions}
\newtheorem{properties}[thm]{Properties}
\newtheorem{example}[thm]{Example}
\newtheorem{conjecture}[thm]{Conjecture}
\newtheorem{proposition and definition}[thm]{Proposition and Definition}
\numberwithin{equation}{thm}

\newcommand{\pP}{{\mathfrak p}}
\newcommand{\sA}{{\mathcal A}}
\newcommand{\sB}{{\mathcal B}}
\newcommand{\sC}{{\mathcal C}}
\newcommand{\sD}{{\mathcal D}}
\newcommand{\sE}{{\mathcal E}}
\newcommand{\sF}{{\mathcal F}}
\newcommand{\sG}{{\mathcal G}}
\newcommand{\sH}{{\mathcal H}}
\newcommand{\sI}{{\mathcal I}}
\newcommand{\sJ}{{\mathcal J}}
\newcommand{\sK}{{\mathcal K}}
\newcommand{\sL}{{\mathcal L}}
\newcommand{\sM}{{\mathcal M}}
\newcommand{\sN}{{\mathcal N}}
\newcommand{\sO}{{\mathcal O}}
\newcommand{\sP}{{\mathcal P}}
\newcommand{\sQ}{{\mathcal Q}}
\newcommand{\sR}{{\mathcal R}}
\newcommand{\sS}{{\mathcal S}}
\newcommand{\sT}{{\mathcal T}}
\newcommand{\sU}{{\mathcal U}}
\newcommand{\sV}{{\mathcal V}}
\newcommand{\sW}{{\mathcal W}}
\newcommand{\sX}{{\mathcal X}}
\newcommand{\sY}{{\mathcal Y}}
\newcommand{\sZ}{{\mathcal Z}}
\newcommand{\A}{{\mathbb A}}
\newcommand{\B}{{\mathbb B}}
\newcommand{\C}{{\mathbb C}}
\newcommand{\D}{{\mathbb D}}
\newcommand{\E}{{\mathbb E}}
\newcommand{\F}{{\mathbb F}}
\newcommand{\G}{{\mathbb G}}
\newcommand{\HH}{{\mathbb H}}
\newcommand{\I}{{\mathbb I}}
\newcommand{\J}{{\mathbb J}}
\renewcommand{\L}{{\mathbb L}}
\newcommand{\M}{{\mathbb M}}
\newcommand{\N}{{\mathbb N}}
\renewcommand{\P}{{\mathbb P}}
\newcommand{\Q}{{\mathbb Q}}
\newcommand{\R}{{\mathbb R}}
\newcommand{\SSS}{{\mathbb S}}
\newcommand{\T}{{\mathbb T}}
\newcommand{\U}{{\mathbb U}}
\newcommand{\V}{{\mathbb V}}
\newcommand{\W}{{\mathbb W}}
\newcommand{\X}{{\mathbb X}}
\newcommand{\Y}{{\mathbb Y}}
\newcommand{\Z}{{\mathbb Z}}
\newcommand{\id}{{\rm id}}
\newcommand{\rank}{{\rm rank}}
\newcommand{\END}{{\mathbb E}{\rm nd}}
\newcommand{\End}{{\rm End}}
\newcommand{\Hom}{{\rm Hom}}
\newcommand{\Hg}{{\rm Hg}}
\newcommand{\tr}{{\rm tr}}
\newcommand{\Cor}{{\rm Cor}}
\newcommand{\GL}{\mathrm{GL}}
\newcommand{\SL}{\mathrm{SL}}
\newcommand{\Aut}{\mathrm{Aut}}
\newcommand{\Sym}{\mathrm{Sym}}
\newcommand{\DD}{\mathbf{D}}
\newcommand{\EE}{\mathbf{E}}
\newcommand{\Gal}{\mathrm{Gal}}
\newcommand{\GSp}{\mathrm{GSp}}
\newcommand{\Spf}{\mathrm{Spf}}
\newcommand{\Spec}{\mathrm{Spec}}
\newcommand{\SU}{\mathrm{SU}}
\newcommand{\Res}{\mathrm{Res}}
\newcommand{\Rep}{\mathrm{Rep}}
\newcommand{\ord}{\mathrm{ord}}
\thanks{This work is supported by the SFB/TR 45 `Periods, Moduli
Spaces and Arithmetic of Algebraic Varieties' of the DFG and
partially supported by the University of Science and Technology of
China.}
\begin{abstract}
We provide a construction of associating a de Rham subbundle to a
Higgs subbundle in characteristic $p$ in the geometric case. As
applications, we obtain a Higgs semistability result and a
$W_2$-unliftable result.
\end{abstract}

\maketitle

\section*{Introduction}
This small note grows out of our efforts to understand the
spectacular work of Ogus-Vologodsky \cite{OV} (see \cite{Sch} for
the logarithmic analogue) on the nonabelian Hodge theory in
characteristic $p$. Let $k$ be an algebraically closed field of
positive characteristic $p$, and $X$ a smooth variety over $k$ which
admits a $W_2(k)$-lifting.  The authors loc. cit. establish a
correspondence between a category of vector bundles with integrable
connections and a category of Higgs bundles over $X$, the objects of
which are subject to certain nilpotent conditions (see Theorem 2.8
loc. cit.). The whole theory is analogous to the one over complex
numbers (see \cite{Si1}). Their construction relies either on the
theory of Azumaya algebra or on a certain universal algebra $\sA$
associated to a $W_2$-lifting of $X$ on which both an integrable
connection and a Higgs field act (see \S2 loc. cit.). The
correspondence is generally complicated. However, there are two
cases where the correspondence is known to be classical: the zero
Higgs field and the geometric case. In the former case this
correspondence is reduced to a classical result of Cartier (see
Remark 2.2 loc. cit. and Theorem 5.1 \cite{Ka}), while in the latter
case the Higgs bundle corresponding to the Gau{\ss}-Manin system of
a geometric family is obtained by taking gradings of the de Rham
bundle with respect to either the Hodge filtration or the conjugate
filtration by the Katz's $p$-curvature formula (see Remark 3.19 loc.
cit. and Theorem 3.2 \cite{Ka1}). Unlike the zero Higgs field case,
an explicit construction of the converse direction in the geometric
case is still unknown. In the complex case this amounts to solving
the Hermitian-Yang-Mills equation (see \cite{Si} for the compact
case), which is transcendental in nature.

The main finding of this note is that one can profit from using the
relative Frobenius of a geometric family, which behaves like the
Hodge metric of the associated variation of Hodge structures over
$\C$ in a certain sense. Indeed, we show that one can use them to
construct a de Rham subbundle from a Higgs subbundle in the
geometric case. Throughout this note $p$ is an odd prime and $n$ is
an integer which is greater than or equal to $p-2$. Our main results
read as follows:
\begin{theorem}\label{theorem 1}
Let $X$ be a smooth scheme over $W=W(k)$ and $f: Y\to X$ a proper
smooth morphism as given in Example \ref{geometric situation}. Let
$(H,\nabla)$ (resp. $(E,\theta)$) be the associated de Rham bundle
(resp. Higgs bundle) of degree $n$ to $f$. Then to any Higgs
subbundle $(G,\theta)\subset (E,\theta)_0$ in characteristic $p$ one
can associate naturally a de Rham subbundle
$(H_{(G,\theta)},\nabla)$ of $(H,\nabla)_0$. For a subsystem of
Hodge bundles in $(E,\theta)_0$ the Cartier-Katz descent of the
associated de Rham subbundle $(H_{(G,\theta)},\nabla)$ to
$(G,\theta)$ is $(G,\theta)$ itself.
\end{theorem}
Our construction is independent of that of Ogus-Vologodsky. It is
interesting to compare it with the inverse Cartier transform of
Ogus-Vologodsky loc. cit. in the situation of the above theorem. The
construction works as well when the base scheme is equipped with a
certain logarithmic structure or defined over $W_{n+1}=W_{n+1}(k)$.
For a Higgs subsheaf of $(E,\theta)_0$, by which we mean a
$\theta$-stable coherent subsheaf in $E_0$, the above construction
yields a de Rham subsheaf of $(H,\nabla)_0$. In the most general
form, the construction works for a Higgs subsheaf of a Higgs bundle
coming from the modulo $p$ reduction of $Gr_{Fil}(H,\nabla)$, where
$(H,Fil^{\cdot},\nabla,\Phi)$ is an object of the Faltings category
$\mathcal{MF}^{\nabla}_{[0,n]}(X)$ (see \S1). We obtain two
applications of the construction.
\begin{proposition}\label{theorem 2}
Assume that $X$ is projective over $W_{n+1}$.
\begin{itemize}
    \item [1)] For a subsystem of Hodge bundles or a Higgs subbundle with zero
Higgs field $(G,\theta)\subset (E,\theta)_0$, one has the slope
inequality $\mu(G)\leq 0$, where $\mu(G)$ is the $\mu$-slope of $G$
with respect to the restriction of an ample divisor of $X$ to $X_0$.
    \item [2)] The following
statements are true:
\begin{itemize}
    \item [i)] Let $g_0: C_0\to X_0$ be a morphism of a smooth projective curve $C_0$ to $X_0$ over $k$ which is liftable to a
morphism $g: C\to X$ over $W_{n+1}$. Then for any Higgs subbundle
$(G,\theta)\subset (E,\theta)_0$, one has $\deg(g_0^*G)\leq 0$.
    \item [ii)] Let $C$ be a smooth projective curve in $X$ over $W_{n+1}$. Then the Higgs bundle $(E,\theta)_0$ is
Higgs semistable with respect to the $\mu_{C_0}$-slope.
\end{itemize}
\end{itemize}
\end{proposition}

Let $F$ be a real quadratic field such that $p$ is inert in $F$. Let
$m\geq 3$ be an integer coprime to $p$. Let $M$ be the smooth scheme
over $W$ which represents the fine moduli functor which associates a
$W$-algebra $R$ a principally polarized abelian surface over $R$
with a real multiplication $\sO_F$ and a full symplectic level
$m$-structure. Let $S_h\subset M_{0}$ be the Hasse locus which is
known to be a $\P^1$-configuration in characteristic $p$ (see
Theorem 5.1 \cite{BG}).
\begin{proposition}\label{normal crossing case}
Let $D$ be an irreducible component in $S_h$. Let $C=\sum C_i$ (may
be empty) be a simple normal crossing divisor in $M_{0}$ such that
$D+C$ is again of simple normal crossing. If the intersection number
$D\cdot C$ is less than or equal to $p-1$, then the curve $D+C$ is
$W_2$-unliftable inside the ambient space $M_{1}$.
\end{proposition}

As a convention, we denote the reduction modulo $p^{i+1}, i\geq 0$
of an object by attaching the subscript $i$. However for a
connection or a Higgs field this rule will not be strictly followed
for simplicity of notation.

\section{The category $\mathcal{MF}_{[0,n]}^{\nabla}(X)$}\label{section on faltings
category} In his study of $p$-adic comparison over a geometric base,
Faltings has introduced the category
$\mathcal{MF}_{[0,n]}^{\nabla}(X)$ in various settings. Its objects
are the strong divisible filtered Frobenius crystals over $X$, which
could be considered as the $p$-adic analogue of a variation of
$\Z$-Hodge structures over a complex algebraic manifold. One shall
be also aware of the fact that Ogus has developed systematically the
category of $F-T$-crystals in the book \cite{O}, which is closely
related to the category $\mathcal{MF}_{[0,n]}^{\nabla}(X)$ (see
particularly \S5.3 loc. cit.).
\subsection{Smooth case}
Let $X$ be a smooth $W$-schme. A small affine subset $U$ of $X$ is
an open affine subscheme $U\subset X$ over $W$ which is \'{e}tale
over $\A_W^d$ \footnote{A small affine open subset in the sense of
Faltings in the $p$-adic comparison in the $p$-adic Hodge theory is
required to be \'{e}tale over $\G_m^d$. Since nowhere in this note
the $p$-adic comparison is used, it suffices to take the above
notion for a small affine subset.}. As $X$ is smooth over $W$ there
exists an open covering consisting of small affine subsets of $X$.
Let $U\subset X$ be a small affine subset. For it one could choose a
Frobenius lifting $F_{\hat U}$ on $\hat U$, the $p$-adic completion
of $U$. An object in $\mathcal{MF}_{[0,n]}^{\nabla}(\hat U)$ (see
Ch. II \cite{Fa}, \S3 \cite{Fa1}) is a quadruple $(H, Fil^{\cdot},
\nabla, \Phi_{F_{\hat U}})$, where
\begin{itemize}
    \item [i)] $(H,Fil^{\cdot})$ is a filtered free $\sO_{\hat U}$-module
    with a basis $e_i$ of $Fil^i, 0\leq i\leq n$ .
    \item [ii)] $\nabla$ is an integrable connection on $H$ satisfying the Griffiths
    transversality:
    $$
    \nabla(Fil^i)\subset Fil^{i-1} \otimes \Omega^1_{\hat U}.
    $$
    \item [iii)] The relative Frobenius is an $\sO_{\hat U}$-linear morphism $\Phi_{F_{\hat U}}: F_{\hat U}^*H\to H$ with the strong
    $p$-divisible property: $\Phi_{F_{\hat U}}(F_{\hat U}^*Fil^i)\subset p^iH$
    and
    $$\sum_{i=0}^{n}\frac{\Phi_{F_{\hat U}}(F_{\hat U}^*Fil^i)}{p^i}=H.$$
    \item [iv)] The relative Frobenius $\Phi_{F_{\hat U}}$ is horizontal
    with respect to the connection $F_{\hat U}^*\nabla$ on $F_{\hat U}^*H$ and
    $\nabla$ on $H$.
\end{itemize}
The filtered-freeness in i) means that the filtration $Fil^{\cdot}$
on $H$ has a splitting such that each $Fil^i$ is a direct sum of
several copies of $\sO_{\hat U}$. The pull-back connection $F_{\hat
U}^*\nabla$ on $F_{\hat U}^*H$ is the composite
\begin{eqnarray*}
   F_{\hat U}^*H=F_{\hat U}^{-1}H\otimes_{F_{\hat U}^{-1}\sO_{\hat U}}\sO_{\hat
U} &\stackrel{F_{\hat U}^{-1}\nabla\otimes id}{\longrightarrow}&
 (F_{\hat U}^{-1}H\otimes F_{\hat U}^{-1}\Omega^1_{\hat
U})\otimes_{F_{\hat U}^{-1}\sO_{\hat U}}\sO_{\hat U} \\
    &=&  F_{\hat U}^*H\otimes F_{\hat U}^*\Omega^1_{\hat U} \stackrel{id\otimes
dF_{\hat U}}{\longrightarrow}F_{\hat U}^*H\otimes \Omega^1_{\hat U}.
\end{eqnarray*}
The horizontal condition iv) is expressed by the commutativity of
the diagram
$$
 \xymatrix{
    F_{\hat U}^*H \ar[d]_{F_{\hat U}^*\nabla} \ar[r]^{\Phi_{F_{\hat U}}} &  H \ar[d]^{\nabla} \\
    F_{\hat U}^*H\otimes
\Omega^1_{\hat U}\ar[r]^{\Phi_{F_{\hat U}}\otimes id} &  H\otimes
\Omega^1_{\hat U}. }
$$
As there is no canonical Frobenius liftings on $\hat U$, one must
know how the relative Frobenius changes under another Frobenius
lifting. This is expressed by a Taylor formula. Let $\hat U=\Spf R$
and $F: R\to R$ a Frobenius lifting. Choose a system of \'{e}tale
local coordinates $\{t_1,\cdots,t_d\}$ of $U$ (namely fix an
\'{e}tale map $U\to \Spec(W[t_1,\cdots,t_d])$). Let $R'$ be any
$p$-adically complete, $p$-torsion free $W$-algebra, equipped with a
Frobenius lifting $F': R'\to R'$ and a morphism of $W$-algebras
$\iota: R\to R'$. Then the relative Frobenius $\Phi_{F'}:
F'^*(\iota^*H)\to \iota^*H$ is the composite
$$
F'^*\iota^*H\stackrel{\alpha}{\cong}
\iota^*F^*H\stackrel{\iota^*\Phi_{F}}{\longrightarrow} \iota^*H,
$$
where the isomorphism $\alpha$ is given by the formula:
$$
\alpha(e\otimes 1)=\sum_{\underline i}\nabla_{\partial}^{\underline
i}(e)\otimes \frac{z^{\underline i}}{\underline{i}!}.
$$
Here $\underline{i}=(i_1,\cdots,i_d)$ is a multi-index, and
$z^{\underline i}=z_1^{i_1}\cdots z_d^{i_d}$ with $z_i=F'\circ
\iota(t_i)-\iota\circ F(t_i), 1\leq i\leq d$, and
$\nabla_{\partial}^{\underline j}=\nabla_{\partial_{
t_1}}^{i_1}\cdots\nabla_{\partial_{t_d}}^{i_d}$.

One defines then the category $\mathcal{MF}_{[0,n]}^{\nabla}(X)$
(see Theorem 2.3 \cite{Fa}). Its object will be denoted again by a
quadruple $(H,Fil^{\cdot},\nabla,\Phi)$. Here $(H,Fil^\cdot,\nabla)$
is a locally filtered free $\sO_X$-module with an integrable
connection satisfying the Griffiths transversality. For each small
affine $U\subset X$ and each choice $F_{\hat U}$ of Frobenius
liftings on $\hat U$, $\Phi$ defines $\Phi_{F_{\hat U}}: F_{\hat
U}^*H|_{\hat U}\to H|_{\hat U}$ such that it, together with the
restriction of $(H,Fil^\cdot,\nabla)$ to $\hat U$, defines an object
in $\mathcal{MF}_{[0,n]}^{\nabla}(\hat U)$.
\begin{example}\label{geometric situation}
Let $f: Y\to X$ be a proper smooth morphism of relative dimension
$n\leq p-2$ between smooth
    $W$-schemes. Assume that the relative Hodge cohomologies $R^if_*\Omega^j_Y, i+j=n$ has no torsion. By Theorem 6.2
    \cite{Fa} \footnote{Faltings loc. cit considered only the $p$-torsion
objects. One obtains this by
    passing to the $p$-adic limit or applying another result of Faltings (see Remark pp. 124 \cite{Fa1}).}, the crystalline direct image $R^nf_{*}(\sO_{Y},d)$ is
    an object in $\mathcal{MF}^{\nabla}_{[0,n]}(X)$.
\end{example}
\subsection{Logarithmic case}
The category $\mathcal{MF}_{[0,n]}^{\nabla}(X)$ has a logarithmic
variant (see IV c) \cite{Fa}, \S4 (f) \cite{Fa90}, and \S3
\cite{Fa1}). A generalization of the Faltings category to a syntomic
fine logarithmic scheme over $W$ can be found in \S2 \cite{T}. We
shall focus only on two special cases: the case of 'having a divisor
at infinity' and the semistable case. In the first case, $X$ is
assumed to be a smooth scheme over $W$, and $D\subset X$ a divisor
with simple normal crossings relative to $W$, i.e. $D=\cup_{i}D_i$
is the union of smooth $W$-schemes $D_i$ meeting transversally. In
the second case, $X$ is assumed to be a regular scheme over $W$ such
that \emph{Zariski} locally there is an \'{e}tale morphism to the
affine space $\A_W^d$ or $\Spec(W[t_1,\cdots,t_{d+1}]/(t_1\cdots
t_{d+1}-p))$ over $W$. We call an open affine subset $U\subset X$
small if there is an \'{e}tale morphism $U\to \A^d_W$ mapping $U\cap
D$ to a union of coordinate hyperplanes (may be empty) in the first
case, and if $U$ satisfies one of the two conditions in the second
case. In each case one associates a natural fine logarithmic
structure to the scheme $X$ such that the structural morphism $X\to
\Spec W$ is log smooth (in the former case one equips $\Spec W$ with
the trivial log structure and in the latter case with the log
structure determined by the closed point of $\Spec W$). See (1.5)
(1) and Examples (3.7) (2) \cite{K}. Note also that the assumptions
made above use the Zariski topology on $X$ instead of the the
\'{e}tale topology as in \cite{K}. The logarithmic crystalline site
is then defined for $X\to \Spec W$ with the above logarithmic
structure (see \S5 loc. cit.).

Let $X\to \Spec W$ be as above. Compared with the definition of
$\mathcal{MF}_{[0,n]}^{\nabla}(X)$ for the smooth case, we shall
take the following modifications for the logarithmic analogue. In
the first case, for a small affine open subset $U\subset X$ a
Frobenius lifting on $\hat U$ shall respect the divisor
$\widehat{U\cap D}\subset \hat U$ (called a logarithmic Frobenius
lifting by Faltings), and $\nabla$ is an logarithmic integrable
connection
$$
\nabla(Fil^i)\subset Fil^{i-1}\otimes \Omega^1_{\hat U}(\log
\widehat{U\cap D}).
$$
In the second case, for an affine open subset $U\subset X$ which
meets the singularities of $X_0$ it is \emph{necessary} to consider
a closed $W$-embedding $i: U\hookrightarrow Z$ in the category of
logarithmic schemes together with a logarithmic Frobenius lifting on
$Z$, by which we mean a Frobenius lifting respecting the logarithmic
structure. In the current special case $Z$ can be chosen to be
smooth over $W$. Write $J$ for the PD-ideal of $i$ and
$D^{\log}_U(Z)$ the logarithmic PD-envelope of $U$ in $Z$ (see
Proposition 5.3 \cite{K}). Denote by $\widehat{D^{\log}_U(Z)}$ the
$p$-adic completion of $D^{\log}_U(Z)$. Then $H$ is a free
$\sO_{\widehat{D^{\log}_U(Z)}}$-module and the decreasing filtration
$Fil^{\cdot}$ on $H$ is compatible with the PD-filtration
$J^{[\cdot]}$ on $\sO_{\widehat{D^{\log}_U(Z)}}$ and is filtered
free (see Page 119 \cite{Fa1}). For the formal logarithmic scheme
$\widehat{D^{\log}_U(Z)}$ let $\Omega^1_{\widehat{D^{\log}_U(Z)}}$
be the sheaf of the formal relative logarithmic differentials on
$\widehat{D^{\log}_U(Z)}$ (see (1.7) \cite{K}). For a choice of a
logarithmic Frobenius lifting $F_Z$ on $Z$ let
$F_{\widehat{D^{\log}_U(Z)}}$ be the induced morphism on
$\widehat{D^{\log}_U(Z)}$. Then by replacing $\hat U$ in the
definition of $\mathcal{MF}^{\nabla}_{[0,n]}(\hat U)$ in \S1 with
$\widehat{D^{\log}_U(Z)}$ we get the description of the local
category $\mathcal{MF}^{\nabla}_{[0,n]}(\widehat{D^{\log}_U(Z)})$.
Taking a small affine covering $\sU=\{U\}$ of $X$ and a family of
closed embeddings $i: U\to Z$ in the second case, one defines the
global category $\mathcal{MF}_{[0,n]}^{\nabla}(X)$ (see \S4 (f)
\cite{Fa90}, \S2 \cite{T}).

One basic example of objects in the category
$\mathcal{MF}_{[0,n]}^{\nabla}(X)$ is provided by the result of
Faltings (Theorem 6.2 \cite{Fa}, Remark \S3, page 124 \cite{Fa1}):
For a $W$-morphism $f: Y\to X$ which is proper, log-smooth and
generically smooth at infinity, if the relative Hodge cohomology
$R^if_{*}\Omega^j_{Y,\log},i+j=n$ has no torsion, then the direct
image $R^nf_*(\sO_{Y},d)$ of the constant filtered Frobenius
logarithmic crystal of $Y$ is an object in the category
$\mathcal{MF}_{[0,n]}^{\nabla}(X)$.

One needs also a logarithmic version of the Taylor formula for the
same purpose as in the smooth case. For that we refer the reader to
the formula (6.7.1) in \cite{K}. In the semistable case we make it
more explicitly as follows. For $U=\Spec R$ \'{e}tale over $\Spec
W[t_1,\cdots,t_{d+1}]/(\prod_{1\leq i\leq d+1}t_i-p)$, choose a
surjection $R'\twoheadrightarrow R$ of $W$-algebras with $R'$ log
smooth over $W$ and a logarithmic Frobenius lifting $F'$ on the
$p$-adic completion $\hat {R'}$. Assume $\{d\log x_1,\cdots,d\log
x_{r}\}$ forms a basis for $\Omega^{1}_{\log}(R')$. For another
choice $R''\twoheadrightarrow R$ with the following commutative
diagram
$$
 \xymatrix{
  R' \ar[d]_{\iota} \ar[r]^{} & R        \\
  R'' \ar[ur]_{}                     },
$$
and $F'': \hat{R''}\to \hat {R''}$ a logarithmic Frobenius lifting,
we let $$u_i=F''\circ \iota(x_i)/\iota\circ F'(x_i), 1\leq i\leq r$$
and $\nabla^{\log}_{\partial_{ x_i}}$ be the differential operator
defined in Theorem 6.2 (iii) \cite{K}. Then $\alpha:
F''^*(\iota^*H)\to \iota^*F'^*H$ given by the Taylor formula
$$
\alpha(e\otimes
1)=\sum_{\underline{i}=(i_1,\cdots,i_l,\cdots,i_{r})\in
\N^{r}}(\prod_{1\leq i\leq r,0\leq j\leq
i_{l}}(\nabla^{\log}_{\partial_{x_i}}-j)(e))\otimes (\prod_{1\leq
i\leq r}\frac{(u_i-1)^{i_l}}{i_l!})
$$
is an isomorphism.

\begin{remark}
The analogue of the category $\mathcal{MF}^{\nabla}_{[0,n]}(X)$
exists when $X$ is smooth over a truncated Witt ring. In this case
$H$ is of $p$-torsion. So the formulation of strong divisibility as
stated in iii) has to be modified. See \S2 c)-d) \cite{Fa}. In the
logarithmic case one finds in \S2.3 \cite{T} the corresponding
modification. Other conditions of the category can be obtained by
taking reduction directly. In the following we shall abuse the
notions of $\mathcal{MF}^{\nabla}_{[0,n]}(X)$ in the case of (log)
smooth $X$ over $W$ for the corresponding category in the case of
(log) smooth over a truncated Witt ring.
\end{remark}
\section{The construction}\label{section on construction}
Let $X$ be a smooth scheme over $W$ and
$(H,Fil^{\cdot},\nabla,\Phi)$ an object in
$\mathcal{MF}^{\nabla}_{[0,n]}(X)$. Let
$(E,\theta)=Gr_{Fil}(H,\nabla)$ be the associated Higgs bundle and
$(G,\theta)\subset (E,\theta)_0$ a Higgs subbundle. We start with a
description of a construction of the de Rham subbundle
$(H_{(G,\theta)},\nabla)\subset (H,\nabla)_0$. We first notice that
there is a natural isomorphism of $\sO_{X_i}$-modules:
$$\frac{1}{[p^i]}: p^iH/p^{i+1}H\to H_0.$$ This follows from the
snake lemma applied to the commutative diagram of
$\sO_{X_i}$-modules:
$$
\xymatrix{ 0\ar[r]&pH/p^{i+1}H\ar@{=}[d]\ar[r]&H/p^{i+1}H\ar@{=}[d]\ar[r]^{p^i}&p^iH/p^{i+1}H\ar[r]&0 &\\
0\ar[r]&pH/p^{i+1}H\ar[r]&H/p^{i+1}H\ar[r]&H/pH\ar[u]_{p^i}\ar[r]&0&
}
$$

\subsection{Local construction}
Take $U\in \sU$ and a Frobenius lifting $F_{\hat U}: \hat U\to \hat
U$. So we get an object $(H_U,Fil^{\cdot}_U,\nabla_U,\Phi_{F_{\hat
U}})\in \mathcal{MF}^{\nabla}_{[0,n]}(\hat U)$ by the restriction.
For simplicity of notation, we omit the appearance of $U$ in this
paragraph. Consider the composite
$$
\frac{\Phi_{F_{\hat U}}}{[p^i]}: F_{\hat U}^*Fil^iH
\stackrel{\Phi_{F_{\hat U}}}{\longrightarrow}p^iH\twoheadrightarrow
p^iH/p^{i+1}H\stackrel{\frac{1}{[p^i]}}{\longrightarrow}H_0.
$$
By the property that $\Phi_{F_{\hat U}}(F_{\hat
U}^*Fil^{i+1})\subset p^{i+1}H$ the above map factors through the
quotient
$$
F_{\hat U}^*Fil^iH\twoheadrightarrow F_{\hat U}^*Fil^iH/F_{\hat
U}^*(Fil^{i+1}H+pFil^iH).
$$
By the filtered-freeness in i) one has $Fil^{i+1}H\cap
pFil^iH=pFil^{i+1}H$. So one obtains an isomorphism
$$
Fil^iH/Fil^{i+1}H+pFil^iH\cong E^{i,n-i}/pE^{i,n-i},
$$
hence an $\sO_{U_0}$-morphism
$$
\frac{\Phi_{F_{\hat U}}}{[p^i]}: F_{U_0}^*(E^{i,n-i})_0\to H_0.
$$
It follows from the strong $p$-divisibility (see \S1 iii)) that
the map
$$
\tilde{\Phi}_{F_{\hat U}}:=\sum_{i=0}^{n}\frac{\Phi_{F_{\hat
U}}}{[p^i]}: F_{U_0}^*E_0\to H_0
$$
is an isomorphism. For another choice of Frobenius lifting $F'_{\hat
U}$ over $\hat U$, write $z_i:=F_{\hat U}(t_i)-F'_{\hat U}(t_i)$. We
have the following
\begin{lemma}\label{taylor formula using higgs field over the same open
set} For a multi-index $\underline{j}=(j_1,\cdots,j_d)$, write
$|\underline j|=\sum_{l=1}^dj_l$ and $\theta_{\partial}^{\underline
j}=\theta_{\partial_{
t_1}}^{j_1}\cdots\theta_{\partial_{t_d}}^{j_d}$. Then for a local
section $e\in (E^{i,n-i})_0(U_0)$, one has the formula
$$
\frac{\Phi_{F_{\hat U}}}{[p^i]}(e\otimes 1)-\frac{\Phi_{F'_{\hat
U}}}{[p^i]}(e\otimes 1)=\sum_{|\underline
j|=1}^{i}\frac{\Phi_{F'_{\hat U}}}{[p^{i-|\underline
j|}]}(\theta_{\partial}^{\underline j}(e)\otimes
1)\otimes\frac{z^{\underline j}}{p^{|\underline j|}\underline{j}!}
$$
\end{lemma}
\begin{proof}
First of all, as each $z_j, 1\leq j\leq d$ is divisible by $p$ and
$i\leq n\leq p-2$, $\frac{z^{\underline j}}{p^{|\underline
j|}\underline{j}!}$ in the above formula is a well-defined element
in $\sO_{U_0}$. Let $\tilde e\in Fil^iH$ be a lifting of $e$.
Applying the Taylor formula over $\sO_{\hat U}$ in the situation
that $R'=R$ and $\iota=id$, we get
$$
\Phi_{F_{\hat U}}(\tilde e\otimes1)=\sum_{|\underline
j|=0}^{\infty}\Phi_{F_{\hat
U}^{\prime}}(\nabla_{\partial}^{\underline j}(\tilde e)\otimes
1)\otimes\frac{z^{\underline j}}{\underline{j}!}.
$$
We observe $\ord_p(\frac{p^{|\underline j|}}{\underline{j}!})\geq
p-1$ for $|\underline j|\geq p$ and $\ord_p(\frac{p^{|\underline
j|}}{\underline{j}!})=|\underline j|$ for $|\underline j|\leq p-1$.
The the above formula can written as
$$
\Phi_{F_{\hat U}}(\tilde e\otimes1)-\Phi_{F'_{\hat U}}(\tilde
e\otimes1)=\sum_{|\underline j|=1}^{i}\Phi_{F_{\hat
U}^{\prime}}(\nabla_{\partial}^{\underline j}(\tilde e)\otimes
1)\otimes\frac{z^{\underline j}}{\underline{j}!}+\sum_{|\underline
j|\geq i+1}\Phi_{F_{\hat U}^{\prime}}(\nabla_{\partial}^{\underline
j}(\tilde e)\otimes 1)\otimes\frac{z^{\underline
j}}{\underline{j}!}.
$$
As $i+1\leq p-1$, the above estimation on the $p$-adic valuation
implies that the second term in the right side is an element in
$p^{i+1}H$. By the Griffiths transversality,
$\nabla_{\partial}^{\underline j}(\tilde e)\in Fil^{i-|\underline
j|}H$. Write $\tilde e_0=\tilde e\mod p$. Thus we have the following
formula which takes value in $H_0(U_0)$:
$$
\frac{\Phi_{F_{\hat U}}}{[p^i]}( e\otimes1)-\frac{\Phi_{F'_{\hat
U}}}{[p^i]}( e\otimes1)=\sum_{|\underline
j|=1}^{i}\frac{\Phi_{F_{\hat U}^{\prime}}}{[p^{i-|\underline
j|}]}(\nabla_{\partial}^{\underline j}(\tilde e_0)\otimes
1)\otimes\frac{z^{\underline j}}{p^{|\underline j|}\underline{j}!}.
$$
Regarding $\frac{\Phi_{F_{\hat U}^{\prime}}}{[p^{i-|\underline
j|}]}$ as a morphism between sheaves of abelian groups
$$
F_{U_0}^{-1}(Fil^{i-|\underline j|}H)_0\to H_0,
$$
one has the following commutative diagram:
$$
\xymatrix{
F_{U_0}^{-1}(Fil^{i}H)_0\ar[rr]^{F_{U_0}^{-1}\nabla_{\partial}^{\underline
j}}\ar[d]_{pr}& &F_{U_0}^{-1}(Fil^{i-|\underline
j|}H)_0\ar[d]^{pr}\ar[rr]^{\frac{\Phi_{F_{\hat U}^{\prime}}}{[p^{i-|\underline j|}]}} && H_0 \\
   F_{U_0}^{-1}(E^{i,n-i})_0\ar[rr]_{F_{U_0}^{-1}\theta_{\partial}^{\underline j}} &&   F_{U_0}^{-1}(E^{i-|\underline j|,n-i+|\underline j|})_0\ar[urr]_{\frac{\Phi_{F_{\hat U}^{\prime}}}{[p^{i-|\underline j|}]}} &&}
$$
It implies that in the previous formula the connection can be
replaced by the Higgs field. Hence the lemma follows.
\end{proof}
\begin{proposition}\label{taylor formula for tilde phi}
Notation as above. For a local section $e$ of $E_0(U_0)$, one has
the following formula:
$$
\tilde{\Phi}_{F_{\hat U}}(e\otimes 1)-\tilde{\Phi}_{F'_{\hat
U}}(e\otimes 1)=\sum_{|\underline j|=1}^{n}\tilde{\Phi}_{F'_{\hat
U}}(\theta_{\partial}^{\underline j}(e)\otimes 1)\otimes
\frac{z^{\underline j}}{p^{|\underline j|}\underline{j}!}.
$$
\end{proposition}
\begin{proof}
Write $e=\sum_{i=0}^{n}e_i$ with $e_i\in (E^{i,n-i})_0$. Lemma
\ref{taylor formula using higgs field over the same open set}
implies
$$
\tilde{\Phi}_{F_{\hat U}}(e\otimes 1)-\tilde{\Phi}_{F'_{\hat
U}}(e\otimes 1)=\sum_{i=0}^{n}\sum_{|\underline
j|=1}^{i}\frac{\Phi_{F'_{\hat U}}}{[p^{i-|\underline
j|}]}(\theta_{\partial}^{\underline j}(e_i)\otimes
1)\otimes\frac{z^{\underline j}}{p^{|\underline j|}\underline{j}!}.
$$
As $\theta_{\partial}^{\underline
j}(e)=\sum_{i=0}^{n}\theta_{\partial}^{\underline j}(e_i)$ and
$\theta_{\partial}^{\underline j}(e_i)=0$ for $|\underline j|\geq
i+1$, the above summation is equal to
$$
\sum_{|\underline j|=1}^{n}[\sum_{i=|\underline
j|}^{n}\frac{\Phi_{F'_{\hat U}}}{[p^{i-|\underline
j|}]}(\theta_{\partial}^{\underline j}(e_i)\otimes
1)]\otimes\frac{z^{\underline j}}{p^{|\underline j|}\underline{j}!}
=\sum_{|\underline j|=1}^{n}\tilde{\Phi}_{F'_{\hat
U}}(\theta_{\partial}^{\underline j}(e)\otimes 1)\otimes
\frac{z^{\underline j}}{p^{|\underline j|}\underline{j}!}.
$$
\end{proof}

The above proposition justifies the following
\begin{definition}\label{local associated bundle}
For the Higgs subbundle $(G,\theta)\subset (E,\theta)_0$, the
locally associated subbundle $S_{U_0}(G)\subset H_0$ over
$U_0\subset X_0$ is defined to be $\tilde{\Phi}_{F_{\hat
U}}(G_{U_0})$, where $U$ is a small affine subset of $X$ with the
closed fiber $U_0$ and $F_{\hat U}$ is a Frobenius lifting over
$\hat U$.
\end{definition}

\subsection{Gluing}
Take $U,V\in \sU$, and Frobenius liftings $F_{\hat U}, F_V,
F_{\widehat{U\cap V}}$ on $\hat U,\hat V,\widehat{U\cap U}$
respectively. We are going to show the following equality of
subbundles in $H_0|_{U_0\cap V_0}$:
$$S_{U_0}(G)|_{U_0\cap V_0}=S_{U_0\cap
V_0}(G)=S_{V_0}(G)|_{U_0\cap V_0}.$$ The following lemma is a
variant of Lemma \ref{taylor formula using higgs field over the same
open set}:
\begin{lemma}\label{taylor formula using higgs field over the inclusion of open
sets} Write $z_i=F_{\hat U}\circ \iota(t_i)-\iota\circ
F_{\widehat{U\cap V}}(t_i)$, where $\iota:\widehat{ U\cap
V}\hookrightarrow \hat U$ is the natural inclusion. Then for a local
section $e\in (E^{i,n-i})_0(U_0)$, one has the formula
$$
\iota_0^*[\frac{\Phi_{F_{\hat U}}}{[p^i]}(e\otimes
1)]-\frac{\Phi_{F_{\widehat{U\cap V}}}}{[p^i]}[\iota_0^*(e)\otimes
1]=\sum_{|\underline j|=1}^{i}\frac{\Phi_{F_{\widehat{U\cap
V}}}}{[p^{i-|\underline
j|}]}(\iota_0^*[\theta_{\partial}^{\underline j}(e)]\otimes
1)\otimes\frac{z^{\underline j}}{p^{|\underline j|}\underline{j}!}
$$
\end{lemma}
\begin{proof}
The proof is the same as in Lemma \ref{taylor formula using higgs
field over the same open set} except that we shall apply the Taylor
formula in the situation that $R'$ is the one with
$\Spf(R')=\widehat {U\cap V}$, $F'=F_{\widehat{U\cap V}}$ and
$\iota: R\to R'$ is the one induced by the natural inclusion.
\end{proof}
A formula similar to that of Proposition \ref{taylor formula for
tilde phi} shows that $S_{U_0}(G)|_{U_0\cap V_0}=S_{U_0\cap
V_0}(G)$. By symmetry we have also the second half equality. The
open covering $\sU$ of $X$ gives rise to an open covering $\sU_0$ of
$X_0$ by reduction modulo $p$. Thus we glue the locally associated
bundles $\{S_{U_0}(G)\}_{U_0\in \sU_0}$ into a subbundle
$H_{(G,\theta)}\subset H_0$, which we call the \emph{associated
subbundle to} $(G,\theta)$. We remark that the construction is
independent of the choice of a small affine open covering $\sU$ of
$X$ as we can always refine such a covering and Lemma \ref{taylor
formula using higgs field over the inclusion of open sets} shows the
invariance of the construction under a refinement.
\subsection{Horizontal property}
We ought to show the associated subbundle $H_{(G,\theta)}\subset H$
is actually $\nabla$-invariant. Let $F_{\hat U}: \hat U\to \hat U$
be a Frobenius lifting over $\hat U$. Then one can write
$\frac{\partial F_{\hat U}}{\partial t_j}=pf_j$ for $f_j\in
\sO_{\hat U}$. Here is a lemma
\begin{lemma}\label{commutation formula using higgs field}
For a local section $e\in (E^{i,n-i})_0(U_0)$, one has the formula
$$
\nabla_{\partial_{t_j}}[\frac{\Phi_{F_{\hat
U}}}{[p^{i}]}(e\otimes1)]=\frac{\Phi_{F_{\hat
U}}}{[p^{i-1}]}[\theta_{\partial_{t_j}}(e)\otimes f_{j,0}].
$$
\end{lemma}
\begin{proof}
Let $\tilde e\in Fil^iH_U$ be a lifting of $e$. The horizontal
property iv) yields the following commutation formula
$$
\nabla_{\partial_{t_j}}[\Phi_{F_{\hat U}}(\tilde e\otimes
1)]=\Phi_{F_{\hat U}}[\nabla_{\partial_{t_j}}(\tilde e)\otimes
1]\otimes \frac{\partial F_{\hat U}}{\partial t_j}.
$$
Thus we have a formula in characteristic $p$:
$$
\nabla_{\partial_{t_j}}[\frac{\Phi_{F_{\hat
U}}}{[p^{i}]}(e\otimes1)]=\frac{\Phi_{F_{\hat
U}}}{[p^{i-1}]}[\nabla_{\partial_{t_j}}(e)\otimes 1]\otimes f_{j,0}.
$$
Finally by the same reason as given in the proof of Lemma
\ref{taylor formula using higgs field over the same open set} we
could replace the connection in the right side by the Higgs field,
and hence obtain the lemma.
\end{proof}
\begin{proposition}
The associated subbundle $H_{(G,\theta)}$ to the Higgs subbundle
$(G,\theta)\subset (E,\theta)_0$ is a de Rham subbundle of
$(H,\nabla)_0$.
\end{proposition}
\begin{proof}
As the question is local, it suffices to show the invariance
property of the locally associated subbundle $S_{U_0}(G)\subset
H_0|_{U_0}$. Lemma \ref{commutation formula using higgs field}
implies that for a local section $e\in G(U_0)$,
$$
\nabla_{\partial_{t_j}}[\tilde{\Phi}_{F_{\hat
U}}(e\otimes1)]=\tilde{\Phi}_{F_{\hat
U}}[\theta_{\partial_{t_j}}(e)\otimes f_{j,0}],
$$
which is again an element of $S_{U_0}(G)$ by the $\theta$-invariance
of $G$. As $\{\partial_{t_j}\}_{1\leq j\leq d}$ spans
$\textrm{Der}_{k}(\sO_{U_0},\sO_{U_0})$, we have shown that
$\nabla(S_{U_0}(G))\subset S_{U_0}(G)\otimes \Omega^1_{U_0}$ as
claimed.
\end{proof}

\subsection{Variants}
By examining the above construction, one finds immediately that it
works as well for a smooth scheme $X$ over $W_{n+1}$. Also it is
immediate to see that the same construction applies for a coherent
subobject in $(E,\theta)_0$. Namely, for a Higgs subsheaf of
$(E,\theta)_0$, we obtain a de Rham subsheaf of $(H,\nabla)_0$ from
the construction. A similar construction also works in the
logarithmic case. In the case of having a divisor at infinity, one
simply replaces the Frobenius liftings and the integrable connection
in the above construction with the logarithmic Frobenius liftings
and the logarithmic integrable connection. In the semistable case,
for a closed embedding $U\hookrightarrow Z$ as in \S1.2, we replace
the local operators $\frac{\Phi_{F_{\hat U}}}{[p^i]}$ in the smooth
case with the reduction of the operator $ \displaystyle{
\frac{\Phi_{F_{\widehat{D^{\log}_U(Z)}}}}{[p^i]}}$ modulo the
PD-ideal $J$, and use the Taylor formula of \S1.2 in the proofs. The
resulting construction yields a logarithmic de Rham subsheaf of
$(H,\nabla)_0$ for any logarithmic Higgs subsheaf of $(E,\theta)_0$.
\subsection{Basic properties}
Let $X$ be a smooth (resp. log smooth) scheme over $W$ (resp.
$W_{n+1}$) as above, and $(H,Fil^{\cdot},\nabla,\Phi)\in
\mathcal{MF}^{\nabla}_{[0,n]}(X)$. Let
$(E,\theta)=Gr_{Fil}(H,\nabla)$ be the associated Higgs bundle, and
for a Higgs subbundle $(G,\theta)\subset (E,\theta)_0$,
$(H_{(G,\theta)},\nabla)\subset (H,\nabla)_0$ the associated de Rham
subbundle by the previous construction. It is not difficult to check
the following properties:
\begin{proposition}\label{basic properties}
 The following statements hold:
\begin{itemize}
    \item [i)] The construction is compatible with pull-backs. Namely, for $f$ a morphism between smooth (resp. log smooth) schemes over $W$ (resp. $W_{n+1}$), one has
    $$(H_{f^*(G,\theta)},\nabla)=f^*(H_{(G,\theta)},\nabla).$$
    \item [ii)] The construction is compatible with direct
    sum and tensor product.
\end{itemize}
\end{proposition}

One can make our construction into a functor. First of all, one
makes a category as follows: An object in this category is a Higgs
subbundle of an $(E,\theta)_0$, which is the modulo $p$ reduction of
$Gr_{Fil}(H,\nabla)$ where $(H,Fil^{\cdot},\nabla)$ comes from an
object in $\mathcal{MF}^{\nabla}_{[0,n]}(X)$ for an $n\geq p-2$. The
set of morphisms are required to be inclusions of Higgs subbundles
in the same Higgs bundle $(E,\theta)_0$. One defines the parallel
category on the de Rham side. These categories have direct sums and
tensor products. Proposition \ref{basic properties} ii) says that
the functor respects direct sum and tensor product. Summarizing the
above discussions, we have the following
\begin{theorem}
Let $X$ be as above and $(H,Fil^{\cdot},\nabla,\Phi)$ an object in
$\mathcal{MF}^{\nabla}_{[0,n]}(X)$. Let
$(E,\theta)=Gr_{Fil}(H,\nabla)$ be the associated Higgs bundle. Then
one associates \emph{naturally} a Higgs subbundle of $(E,\theta)_0$
to a de Rham subbundle of $(H,\nabla)_0$.
\end{theorem}
In the following let $X$ be a smooth scheme over $W$ or $W_{n+1}$.
The next result relates our construction in the zero Higgs field
case with the Cartier descent (see Theorem 5.1 \cite{Ka}).
\begin{proposition}\label{corollary on the canonical connection}
If $(G,0)\subset (E,\theta)_0$ is a Higgs subbundle with zero Higgs
field, then one has an isomorphism of vector bundles with integrable
connection
$$
\tilde \Phi:
(F_{X_0}^*G,\nabla_{can})\stackrel{\cong}{\longrightarrow}
(H_{(G,0)},\nabla|_{H_{(G,0)}}),
$$
where $\nabla_{can}$ is the canonical connection associated to a
Frobenius pull-back vector bundle.
\end{proposition}
\begin{proof}
This is a direct consequence of the construction of the subbundle
$H_{(G,0)}$ and the formula in Proposition \ref{taylor formula for
tilde phi} in the case of $\theta=0$. Note also that
$\{\tilde{\Phi}(e_i\otimes 1)\}$, where $\{e_i\}$ runs through a
local basis of $G$, makes an integrable basis of $S_{U_0}(G)$, which
follows directly from the formula in Lemma \ref{commutation formula
using higgs field} in the case of $\theta=0$.
\end{proof}
\subsection{Cartier-Katz descent}
Let $(H,Fil^{\cdot},\nabla,\Phi)$ be a geometric one, namely it
comes from Example \ref{geometric situation}. Then $H_0$ is equipped
with the conjugate filtration $0=F^{n+1}_{con}\subset
F^{n}_{con}\subset\cdots\subset F^0_{con}=H_0$, which is horizontal
with respect to the Gau{\ss}-Manin connection (see \S3 in
\cite{Ka}). For a subbundle $W\subset H_0$ we put
$Gr_{F_{con}}(W)=\bigoplus_{q=0}^{n} \frac{W\cap F_{con}^{q}}{W\cap
F_{con}^{q+1}}$. The $p$-curvature $\psi_{\nabla}$ of $\nabla$
defines the $F$-Higgs bundle
$$
\psi_{\nabla}:Gr_{F_{con}}(H_0)\to Gr_{F_{con}}(H_0)\otimes
F_{X_0}^*\Omega_{X_0}.
$$
As a reminder to the reader, we recall the definition of
\emph{$F$-Higgs bundle}: an $F$-Higgs bundle over a base $C$, which
is defined over $k$, is a pair $(E',\theta')$ where $E'$ is a vector
bundle over $C$, and $\theta'$ is a bundle morphism $E'\to E'\otimes
F_C^*\Omega_C$ with the integral property $\theta'\wedge\theta'=0$.
The following lemma is a simple consequence of Katz's $p$-curvature
formula (see Theorem 3.2 \cite{Ka1}).
\begin{lemma}[Lemma 7.2 \cite{SZZ} \footnote{The quoted lemma deals only with the weight one situation, but the proof works for an arbitrary weight.}] \label{grading of conjugate filtration}
Let $(W,\nabla)$ be a de Rham subbundle of $(H,\nabla)_0$. Then the
$F$-Higgs subbundle
$(Gr_{F_{con}}(W),\psi_{\nabla}|_{Gr_{F_{con}}(W)})$ defines a Higgs
subbundle of $(E,\theta)_0$ by the Cartier descent.
\end{lemma}
We call the above Higgs subbundle the \emph{Cartier-Katz descent} of
$(W,\nabla)$. Back to the discussion on the associated de Rham
subbundle $(H_{(G,\theta)},\nabla)\subset (H,\nabla)_0$ with
$(G,\theta)\subset (E,\theta)_0$. After a terminology of Simpson
(see \cite{Si}) we shall call a Higgs subbundle $(G,\theta)$ with
the property $G=\oplus_{i=0}^{n}(G\cap (E^{i,n-i})_0)$ a
\emph{subsystem of Hodge bundles}. We have the following
\begin{proposition}\label{cartier-katz descent special case}
Let $(G,\theta)\subset (E,\theta)_0$ be a subsystem of Hodge
bundles. Then the Cartier-Katz descent of $(H_{(G,\theta)},\nabla)$
is equal to $(G,\theta)$.
\end{proposition}
\begin{proof}
First we recall that the relative Cartier isomorphism defines an
isomorphism
$$
\sC: Gr_{F_{con}}H_0\stackrel{\cong}{\longrightarrow} F_{X_0}^*E_0.
$$
We need to show that it induces an isomorphism
$$
Gr_{F_{con}}(H_{(G,\theta)})\cong F_{X_0}^*G.
$$
Write $G^{i,n-i}=G\cap (E^{i,n-i})_0$. Then
$G=\oplus_{i=1}^{n}G^{i,n-i}$. Now that over $U_0$ the composite
$$
F^*_{U_0}E^{i,n-i}_0|_{U_0}\stackrel{\frac{\Phi_{F_{\hat
U}}}{[p^i]}}{\longrightarrow}
F_{con}^{n-i}H_0|_{U_0}\twoheadrightarrow Gr_{Fon}^{n-i}H_0|_{U_0}
$$
is the inverse relative Cartier isomorphism $\sC^{-1}|_{U_0}$ over
$U_0$, it follows from the local construction of $H_{(G,\theta)}$
that
$$
\sC^{-1}|_{U_0}(F^*_{U_0}G^{i,n-i}|_{U_0})\stackrel{\cong}{\longrightarrow}
Gr_{Fon}^{n-i}H_{(G,\theta)}|_{U_0}.
$$
This implies the result.
\end{proof}
The above proof implies also the equalities $$H_{(G^{\leq
i},\theta)}=H_{(G,\theta)}\cap F_{con}^{n-i}, 0\leq i\leq n,$$ where
$G^{\leq i}$ is the Higgs subbundle $\oplus_{q\leq i}G^{q,n-q}$ of
$G$.
\begin{remark}\label{grading of hodge filtration}
The grading of $(H_{(G,\theta)},\nabla)$ with respect to the Hodge
filtration defines a Higgs subbundle of $(E,\theta)_0$ which is in
general not $(G,\theta)$. In the case that they are equal and $X$ is
proper over $W$, $(G,\theta)$ defines a $p$-torsion
subrepresentation of $\pi^{arith}_1(X^0)$, the \'{e}tale fundamental
group of the generic fiber $X^0$ of $X$, implied by a result of
Faltings (see Theorem 2.6* \cite{Fa}). A similar remark has appeared
in \S4.6 \cite{OV}.
\end{remark}

\section{Applications}
\subsection{Higgs semistability}
In this paragraph $X$ is assumed to be smooth and projective over
$W_{n+1}$ with connected closed fiber $X_0$ over $k$. Fix an ample
divisor $D$ on $X$. Recall that the $\mu$-slope of a torsion free
coherent sheaf $Z$ on $X_0$ is defined to be
$$
\mu(Z)=\frac{c_1(Z)\cdot D_0^{d-1}}{\rank Z}.
$$
\begin{proposition}\label{nonpositivity of subsystems of Hodge
bundles} Let $(E,\theta)$ be the associated Higgs bundle in the
geometric case, i.e. Example \ref{geometric situation}. Then the
following statements hold:
\begin{itemize}
    \item [i)] For any subsystem of Hodge bundle $(G,\theta)\subset (E,\theta)_0$,
one has $\mu(G)\leq 0$.
    \item [ii)] For any Higgs subbundle $G\subset
E_0$ with zero Higgs field, it holds that $\mu(G)\leq 0$.
\end{itemize}
\end{proposition}
\begin{proof}
Assume that there exists a subsystem of Hodge bundles $(G,\theta)$
with positive $\mu$-slope in $(E,\theta)_0$. Take such one with the
largest slope. By the proof of Proposition \ref{cartier-katz descent
special case}, one has an isomorphism
$Gr_{F_{con}}H_{(G,\theta)}\cong F_{X_0}^*G$, and consequently the
equalities $\mu(H_{(G,\theta)})=\mu(F_{X_0}^*G)=p\mu(G)$. Then the
observation in Remark \ref{grading of hodge filtration} says that
$Gr_{Fil}(H_{(G,\theta)},\nabla)$ gives a subsystem of Hodge bundles
of $(E,\theta)_0$ of slope $p\mu(G)>\mu(G)$, a contradiction. Hence
i) follows. Now assume the existence of a Higgs subbundle $(G,0)$
with positive $\mu$-slope. By Corollary \ref{corollary on the
canonical connection}, the associated de Rham subbundle
$H_{(G,0)}\subset H_0$ is isomorphic to $F_{X_0}^*G$, whose
$\mu$-slope is equal to $p\mu(G)>0$. Then
$Gr_{Fil}(H_{(G,0)},\nabla)$ gives rise to a subsystem of Hodge
bundles with positive $\mu$-slope, which contradicts i).
\end{proof}
Let $C\subset X$ be a smooth projective curve over $W_{n+1}$. For a
coherent sheaf $Z$ over $X_0$, the $\mu_{C_0}$-slope of $Z$ is
defined to be $\frac{\deg(Z|_{C_0})}{\rank Z}$. Recall that a Higgs
bundle $(E,\theta)$ over $X_0$ is said to be Higgs semistable with
respect to the $\mu_{C_0}$-slope if for any Higgs subbundle
$(F,\theta)\subset(E,\theta)$ the inequality $\mu_{C_0}(F)\leq
\mu_{C_0}(E)$ holds.
\begin{lemma}\label{degree formula}
For any Higgs subbundle $(G,\theta)\subset (E,\theta)_0$, it holds
that $$\det C_0^{-1}(G,\theta)\cong(F_{X}^*\det G,\nabla_{can}).$$
\end{lemma}
\begin{proof}
The functor $C_0^{-1}$ has been further studied in the paper
\cite{LSZ}. It follows from Proposition 5 loc. cit. that
$C_0^{-1}(G,\theta)$ is isomorphic to the exponential twisting of
$F_X^{*}G$. Precisely, it is obtained by gluing the local flat
bundles $(F_{U}^*G|_{U},\nabla_{can}+\frac{dF_{\hat
U}}{p}F_U^*\theta|_U)$ via the gluing functions
$\exp[h_{UV}(F_{U\cap V}^*\theta)]F_{U\cap V}^*M_{U\cap V}$. Here
$U,V$ are two open subsets of $X$, $M_{U\cap V}$ is the transition
function of two local bases of $G$ over $U\cap V$, $h_{UV}$ is the
derivation measuring the difference of two Frobenius liftings. Due
to the fact that $\theta$ is nilpotent, the determinant of the
exponential twisting is simply the identity. Therefore $\det
C_0^{-1}(G,\theta)$ is isomorphic to $(F_{X}^*\det G,\nabla_{can})$.
\end{proof}
\begin{proposition}
Let $(E,\theta)$ be the associated Higgs bundle to an object in
$\mathcal{MF}_{[0,n]}^{\nabla}(X)$. Then the following statements
hold:
\begin{itemize}
    \item [i)] Let $g_0: C_0\to X_0$ be a morphism from a smooth projective curve $C_0$ to $X_0$ over $k$ which is liftable to a
morphism $g: C\to X$ over $W_{n+1}$. Then for any Higgs subbundle
$(G,\theta)\subset (E,\theta)_0$, one has $\deg(g_0^*G)\leq 0$.
    \item [ii)] Let $C$ be a smooth projective curve in $X$ over $W_{n+1}$. The Higgs bundle $(E,\theta)_0$ is
Higgs semistable with respect to the $\mu_{C_0}$-slope.
\end{itemize}
\end{proposition}
\begin{proof}
Let $(E,\theta)|_{C}$ be the Higgs bundle pulled back via $g$. We
use similar notions for the pull-backs via $g_0$. The pull-back of
$H$ via $g$ gives an object in $\mathcal{MF}^{\nabla}_{[0,n]}(C)$.
Assume the existence of a Higgs subbundle $(G,\theta)\subset
(E,\theta)_0$ satisfying $\deg(G|_{C_0})>0$. By Lemma \ref{degree
formula}, it holds that
$$
\deg Gr_{Fil}C_0^{-1}(G,\theta)|_{C_0}= \deg
C_0^{-1}(G,\theta)|_{C_0}= \deg F_{C_0^{*}}G|_{C_0}
    = p\deg G|_{C_0}.$$
Iterating this process, one obtains Higgs subbundles in $E_0|_{C_0}$
with arbitrary large degrees, which is impossible. Hence the result
i) follows. The proof of ii) is similar by replacing the degree in
the previous argument with the $\mu_{C_0}$-slope.
\end{proof}

\begin{remark}
In the above result i), it is natural to make the liftability
assumption on $C_0$. The example of Moret-Bailly (see \cite{Mo})
shows that over a $W_2$-unliftable curve in the moduli space of
principal polarized abelian surfaces in characteristic $p$ the Higgs
bundle of the restricted universal family contains a Higgs line
bundle with positive degree. In the case $X$ being a curve, the assumption on
$p$ made in Proposition 4.19 \cite{OV} reads $n(\rank
E-1)\max\{2g-2,1\}\leq p-2$ where $g$ is the genus of $X_0$. The
above result ii) removes the dependence of $p$ on the genus as well
as the rank. 
\end{remark}

\subsection{$W_2$-unliftability}
Let $F$ be a real quadratic field with the ring of integers $\sO_F$.
Assume $p$ is inert in $F$. Fix an integer $m\geq 3$, coprime to
$p$. Let $M$ be the moduli scheme over $W$, and $S_h$ the Hasse
locus of $M_{0}$ as described in the introduction. Let $Z_0\subset
S_h$ be a curve with simple normal crossing. It is said to be
$W_2$-liftable inside $M_{1}$ if there exists a semistable curve
$Z_1\subset M_{1}$ over $W_2$ such that its closed fiber is $Z_0$.
In the following the $W_2$-liftability means always the
$W_2$-liftability inside $M_{1}$. It is interesting to know whether
the components in $S_h$ are liftable to $W_2$. The $W_2$-liftability
on the whole or part of the $\P^1$-configuration $S_h$ is in fact a
subtle problem. On the one hand, it shall be more or less well known
that each component of $S_h$ is $W_2$-unliftable. On the other hand,
a result of Goren (see Theorem 2.1 \cite{G}) implies that the whole
configuration is $W_2$-liftable if the zeta value $\zeta_{F}(2-p)$
is a non $p$-adic integer. Our partial result on this question is
the following
\begin{proposition}\label{thm 3}
Let $D$ be a component in $S_h$. Let $C=\sum C_i$ (may be empty) be
a simple normal crossing divisor in $M_{0}$ such that $D+C$ is again
a simple normal crossing divisor. If $D\cdot C\leq p-1$, then the
curve $D+C$ is $W_2$-unliftable.
\end{proposition}
Before giving the proof, we introduce several notations. Let $f:
X\to M$ be the universal abelian scheme. Let $(E=E^{1,0}\oplus
E^{0,1},\theta)$ be the Higgs bundle of $f$, where
$E^{1,0}=f_*\Omega^1_{X|M}$ and $E^{0,1}=R^1f_*\sO_{X}$. It is known
that $(E,\theta)$ has a decomposition under the $\sO_F\otimes
W$-action in the form
$$
(E,\theta)=(E_1,\theta_1)\oplus (E_2,\theta_2),
$$
where $E^{1,0}_i=\sL_i$ and $E^{0,1}_i=\sL_i^{-1}$ for $i=1,2$. It
is also known that either $\sL_1|_{D}\simeq \sO_{\P^1}(-1), \
\sL_2|_{D}\simeq \sO_{\P^1}(p)$ or $\sL_1|_{D}\simeq \sO_{\P^1}(p),
\ \sL_2|_{D}\simeq \sO_{\P^1}(-1)$ holds for any component $D$ in
$S_h$. One has a description of the Higgs field of the Higgs bundle
associated to the restricted universal family to $D$: In the former
case, $\theta_2: \sL_2|_{D}\to \sL_2^{-1}|_{D}\otimes \Omega^1_{D}$
is zero for the reason of degree, and $\theta_1: \sL_1|_{D}\to
\sL_1^{-1}|_{D}\otimes \Omega^1_{D}$ can be shown to be an
isomorphism (we shall not use this fact in the following argument).
The properties of $\theta_1$ and $\theta_2$ are exactly exchanged in
the latter case. Put the log structure on $Z_{0}=D+C$ by its
components and the trivial log structure on $\Spec k$. Let
$\Omega^1_{\log}(Z_{0}/k)$ be the sheaf of log differentials (see
(1.7) \cite{K}). It is locally free $\sO_{Z_0}$-module of rank one.
The family $f_0$ restricts to a family $f_0: Y_0\to Z_{0}$. With the
pull-back logarithmic structure on $Y_0$, $f_0$ is log smooth. So
one forms the logarithmic de Rham bundle $(H,\nabla)$ of $f_0$,
which is by definition the first hypercohomology of the relative
logarithmic de Rham complex. The relative Hodge filtration on the
complex degenerates at $E_1$, thus one forms the logarithmic Higgs
bundle over $Z_0$:
$$
\eta: F\to F\otimes \Omega^1_{\log}(Z_0/k),
$$
where $F=F^{1,0}\oplus F^{0,1}$ with $F^{1,0}=\sL_1|_{Z_0}\oplus
\sL_2|_{Z_0}$ and $F^{0,1}=\sL_{1}^{-1}|_{Z_0}\oplus
\sL_2^{-1}|_{Z_0}$.
\begin{proof}
Now we assume $Z_0$ lifts a semistable curve $Z\subset M_{1}$
over $W_2$. We equip $Z$ with the log structure determined by the
divisor $Z_0$ and $\Spec W_2$ with the one by $\Spec k$. We can
assume that $$\sL_1|_{D}\simeq \sO_{\P^1}(-1), \ \sL_2|_{D}\simeq
\sO_{\P^1}(p).$$ Over the open subset $D-D\cap C$, $\eta|_{D}$
coincides with the Higgs bundle coming from the restricted universal
family to $D$. So by the above discussion, $\eta|_D(\sL_2|_{D})=0$.
Consider the following coherent subsheaf in $F$: Take the subsheaf
$$
\sL_2|_{D}\otimes \sO_{D}(-D\cap C)\subset \sL_2|_{D}
$$
over $D$ and take the zero sheaf over $C$, considered as the
subsheaf of $\sL_2|_{C}$. They glue into a subsheaf $\sL$ of
$\sL_2|_{Z_0}\subset F$ over $Z_0$ since over a small open
neighborhood of any point $P\in D\cap C$, $\sL_2|_{D}\otimes
\sO_{D}(-D\cap C)$ has a local basis vanishing at $P$. Note that
$\sL$ is a Higgs subsheaf of $F$. In fact the Higgs field $\eta$
acts on $\sL$ trivially by construction. Then the construction of
\S2 in the semistable case applies. So $(\sL,0)\subset (F,\eta)$
gives rise to a de Rham subsheaf $H_{(\sL,0)}\subset (H,\nabla)$.
Note that $H_{(\sL,0)}|_{D}$ is isomorphic to
$F_{D}^*(\sL_2|_{D}\otimes \sO_{D}(-D\cap C))$ and hence has degree
$p(p-D\cap C)$. We have a short exact sequence from the Hodge
filtration:
$$
0\to \sL_1|_{D}\oplus \sL_2|_{D}\to H|_{D} \to \sL_1^{-1}|_{D}\oplus
\sL_2^{-1}|_{D}\to 0.
$$
Then we first consider the composite
$$H_{(\sL,0)}|_{D}\subset
H|_{D}\twoheadrightarrow \sL_1^{-1}|_{D}\oplus \sL_2^{-1}|_{D}.
$$
As $\deg \sL_1^{-1}|_{D} =1$ and $\deg \sL_2^{-1}|_{D} =-p$ are both
smaller than $\deg H_{(\sL,0)}|_{D}$, one has the factorization
$$
H_{(\sL,0)}|_{D}\subset \sL_1|_{D}\oplus \sL_2|_{D}\subset H|_{D}.
$$
{\bf Case 1:} $D\cap C\leq p-2$. Again for the reason of degree, the
above nontrivial map is impossible.

{\bf Case 2:} $D\cap C= p-1$. In this case, one obtains the equality
$$
H_{(\sL,0)}|_{D}=\sL_2|_{D}.
$$
This is also impossible because of the semilinearity of the relative
Frobenius. For a small affine $U\subset Z$ whose modulo $p$
reduction is $U_0\subset D-D\cap C$ and a Frobenius lifting $F_{\hat
U}$, the local operator $\frac{\Phi_{F_{\hat U}}}{[p]}$ maps a local
section in $\sL_2|_{D}(U_0)$ to a local section in
$\sL_1|_{D}(U_0)$. As $\sL_2|_{D}\otimes \sO_{D}(-D\cap C)\subset
\sL_2|_{D}$, it is impossible to have $H_{(\sL,0)}(U_0)\subset
\sL_2|_{D}(U_0)$ by the construction of $H_{(\sL,0)}$.
\end{proof}

\end{document}